\documentclass[12pt,a4paper]{amsart}
\usepackage{amscd,amsfonts,amsmath,amsthm,amssymb,verbatim,mathrsfs}
\usepackage[dvips]{graphicx}
\usepackage{graphics,hyperref}
\usepackage{enumitem}
\usepackage{psfrag}
\usepackage{pstricks}
\usepackage{latexsym}
\usepackage{multicol}

\psset{unit=0.7cm,linewidth=0.8pt,arrowsize=2.5pt 4}
\newpsstyle{fatline}{linewidth=1.5pt}
\newpsstyle{fyp}{fillstyle=solid,fillcolor=verylight}
\definecolor{verylight}{gray}{0.97}
\definecolor{light}{gray}{0.9}
\definecolor{medium}{gray}{0.85}

\def\NZQ{\mathbb}               % the font for N,Z,Q,R,C
\def\NN{{\NZQ N}}

\def\ZZ{{\NZQ Z}}

\def\frk{\frak}               % font for "Fraktur"

\def\Phi{{\frk n}}
\def\Phi{{\frk N}}
\def\ub{{\bold u}}
\def\vb{{\bold v}}
\def\wb{{\bold w}}
\def\ab{\mathbf{a}}

\def\tb{{\bold t}}

\def\cb{\mathbf{c}}

\def\opn#1#2{\def#1{\operatorname{#2}}} % to make operators
\opn\chara{char} \opn\length{\ell} \opn\pd{pd} \opn\rk{rk}
\opn\projdim{proj\,dim} \opn\injdim{inj\,dim} \opn\rank{rank}
\opn\depth{depth} \opn\grade{grade} \opn\height{height}
\opn\embdim{emb\,dim} \opn\codim{codim} \opn\sgn{sgn}

\opn\Tr{Tr} \opn\bigrank{big\,rank}
\opn\superheight{superheight}\opn\lcm{lcm}
\opn\trdeg{tr\,deg}%\emph{
\opn\reg{reg} \opn\lreg{lreg} \opn\ini{in} \opn\lpd{lpd}
\opn\size{size}\opn\bigsize{bigsize}
\opn\cosize{cosize}\opn\bigcosize{bigcosize}
\opn\sdepth{sdepth}\opn\sreg{sreg}
\opn\link{link}\opn\fdepth{fdepth}
\opn\Graver{Graver}
\opn\div{div} \opn\Div{Div} \opn\cl{cl} \opn\Cl{Cl} \opn\Cor{Cor}
\opn\Spec{Spec} \opn\Supp{Supp} \opn\supp{supp} \opn\Sing{Sing}
\opn\Ass{Ass} \opn\Min{Min}\opn\Mon{Mon} \opn\dstab{dstab} \opn\astab{astab}
\opn\Ann{Ann} \opn\Rad{Rad} \opn\Soc{Soc} \opn\Gr{Gr}
\opn\Im{Im} \opn\Ker{Ker} \opn\Coker{Coker} \opn\Am{Am}
\opn\Hom{Hom} \opn\Tor{Tor} \opn\Ext{Ext} \opn\End{End}
\opn\Aut{Aut} \opn\id{id} \opn\span{span}

\opn\nat{nat}
\opn\pff{pf}%   \pf exists already
\opn\Pf{Pf} \opn\GL{GL} \opn\SL{SL} \opn\mod{mod} \opn\ord{ord}
\opn\Gin{Gin} \opn\Hilb{Hilb}\opn\sort{sort} \opn\Gale{Gale}
\opn\aff{aff} \opn\conv{conv} \opn\relint{relint} \opn\st{st}   \opn\cone{cone}
\opn\lk{lk} \opn\cn{cn} \opn\core{core} \opn\vol{vol}
\opn\link{link} \opn\star{star}\opn\lex{lex} \opn\Gr{Gr}
\opn\gr{gr}

\def\pot#1#2{#1[\kern-0.28ex[#2]\kern-0.28ex]}

\opn\dirlim{\underrightarrow{\lim}}
\opn\inivlim{\underleftarrow{\lim}}
\def\Implies{\ifmmode\Longrightarrow \else
        \unskip${}\Longrightarrow{}$\ignorespaces\fi}
\def\implies{\ifmmode\Rightarrow \else
        \unskip${}\Rightarrow{}$\ignorespaces\fi}
\def\iff{\ifmmode\Longleftrightarrow \else
        \unskip${}\Longleftrightarrow{}$\ignorespaces\fi}

\let\:=\colon
\newtheorem{Theorem}{Theorem}[section]

\newtheorem{Proposition}[Theorem]{Proposition}
\newtheorem{Remark}[Theorem]{Remark}

\newtheorem{Example}[Theorem]{Example}

\newtheorem{Definition}[Theorem]{Definition}

\let\epsilon\varepsilon
\let\kappa=\varkappa
\makeatletter
\def\qed{\ifhmode\textqed\fi
      \ifmmode\ifinner\quad\qedsymbol\else\dispqed\fi\fi}
\def\textqed{\unskip\nobreak\penalty50
       \hskip2em\hbox{}\nobreak\hfil\qedsymbol
       \parfillskip=0pt \finalhyphendemerits=0}
\def\dispqed{\rlap{\qquad\qedsymbol}}

\opn\dis{dis}
\def\pnt{{\raise0.5mm\hbox{\large\bf.}}}

\opn\Lex{Lex}

\begin{document}

%\begin{frontmatter}

\title{On the dimension of the strongly robust complex for configurations in general position}
\author[1]{ Dimitra Kosta }
\author[2] {Apostolos Thoma}
\author[3] {Marius Vladoiu}
\thanks{Corresponding author: Dimitra Kosta}
\address{Dimitra Kosta, School of Mathematics, University of Edinburgh and Maxwell Institute for Mathematical Sciences, United Kingdom }
\email{D.Kosta@ed.ac.uk}
%\corauth[cor1]{Corresponding author}

\address{Apostolos Thoma, Department of Mathematics, University of Ioannina, Ioannina 45110, Greece}
\email{athoma@uoi.gr}

\address{Marius Vladoiu, Faculty of Mathematics and Computer Science, University of Bucharest, Str. Academiei 14, Bucharest, RO-010014, Romania}
\email{vladoiu@fmi.unibuc.ro}

\subjclass[2020]{ 05E45, 13F65, 13P10, 14M25}
\keywords{ Toric ideals,  Graver basis, cyclic polytopes, indispensable elements, Robust ideals, Simplicial complex}

\begin{abstract}  Strongly robust toric ideals are the toric ideals for which the set of indispensable binomials is the Graver basis. The strongly robust
simplicial complex $\Delta _T$ of a simple toric ideal $I_T$  determines the strongly robust property for
all toric  ideals that have $I_T$ as their bouquet ideal. We prove that  $\dim \Delta_T<\rank(T)$ for configurations in general position,
partially answering a question posed by Sullivant. 
\end{abstract}
\maketitle
%\end{frontmatter}

\section{Introduction}

  A toric ideal is  {\it strongly robust}  if the
following sets are identical: the set of indispensable binomials, any minimal system of binomial generators, any reduced Gr{\"o}bner basis, the universal Gr{\"o}bner basis and the Graver basis (see \cite{S}). Well known classes  of strongly robust ideals are the Lawrence ideals, see \cite[Chapter 7]{St}, and the toric ideals of non pyramidal self dual projective varieties, see \cite{BDR, TV}.
 There are several articles  in the literature studying   strongly robust toric ideals,
\cite{ BBDLMNS, GT,   GP, KTV, KTV2, NR, PTV, PTV2, TV, St,  S}. Strongly robust toric ideals are of  importance in algebraic statistics as they provide examples of toric ideals satisfying the distance reducing property (\cite{AokiTake}). Another noteworthy property of strongly robust ideals generated by quadrics is that they are examples of Koszul algebras \cite{GT}. 

A key concept in understanding the strongly robust property for toric ideals is the notion of a bouquet,
which was developed by Petrovi{\' c} et al.  in \cite{PTV}. Bouquets are connected components of a graph and are of three types: mixed, non-mixed and free. In  \cite{KTV} Kosta et al. indroduced the  strongly robust simplicial complex $\Delta_T$ which characterizes the strongly robust property of toric ideals that have in common the same bouquet ideal $I_T$. In particular, let $I_A$ be a toric ideal with bouquet ideal $I_T$, the ideal $I_A$ is strongly robust if and only
if the set $\omega$ of indices $i$, such that the $i$-th bouquet of $I_A$ is non-mixed,
is a face of  $\Delta_T$, see \cite[Theorem 3.6]{KTV}. Thus, understanding the strongly robust property of toric ideals $I_A$ is equivalent to understanding the strongly robust simplicial complex $\Delta_T$ for
simple toric ideals $I_T$. Simple toric ideals are ideals for which every bouquet is a 
singleton.  For each simple toric ideal $I_T$ for which the strongly robust complex is 
known one can construct infinite classes of strongly robust toric ideals that 
have in common the same bouquet ideal, $I_T$, using the theory of generalized Lawrence matrices developed in \cite[Section 2]{PTV} For details, see Section~\ref{dimHIGH}.

In \cite{PTV}, Petrovi{\' c} et al. observed that a special class of strongly robust toric ideals, the Lawrence ideals  (see ~\cite[Chapter 7]{St}), contain only mixed bouquets and proved (\cite{PTV}[Corollary 4.4] that if every non-free bouquet of a toric ideal $I_A$ is mixed then $I_A$ is strongly robust.  They also constructed several other examples of strongly robust ideals having both mixed and non mixed bouquets, but never all non-mixed.  Motivated by this, they asked whether every strongly robust toric ideal $I_A$ necessarily admits a mixed bouquet. In \cite{S}, Sullivant proved this conjecture for codimension 2 toric ideals by proving that codimension 2 toric ideals have at least two mixed bouquets and reformulated the question as follows: does every strongly robust toric ideal $I_A$ of codimension $r$ have at least $r$ mixed bouquets? Since bouquets preserve codimension, Sullivant's question is equivalent to a question about the dimension of the strongly robust simplicial complex of its bouquet ideal $I_T$: If $s$ is the number of bouquets of $I_A$, is it true that simple toric ideals $I_T$ of codimension $r$ in the polynomial ring of $s$ variables have $\dim \Delta_T<s-r=\rank (T)$? In \cite{KTV2}, Kosta et al. provide a positive resolution to Sullivant’s question in the case of simple toric ideals of monomial curves, by proving that $\dim \Delta_T$ is strictly less than one which is the $\rank (T)$ for $T$ an $1\times n$ matrix defining a monomial curve, $n\ge 3$. In Section~\ref{dimComplex}, we extend this result by giving an affirmative answer to Sullivant’s question for simple toric ideals of configurations in general position with Theorem~\ref{dimension}.

The structure of the paper is the following. In Section~\ref{section:prelim}, we present the necessary prerequisites on simple toric ideals and the strongly robust complex, and establish some of their basic properties that will be used throughout the remaining of the paper.  Section~\ref{dimComplex}, includes the main result of the paper Theorem~\ref{dimension}, which, for configurations in general position, bounds the dimension of the strongly robust complex $\Delta_T$  by the rank of the matrix $T$. Finally, in Section~\ref{dimHIGH} we produce families of examples of strongly robust ideals
with bouquet ideal the ideal of a configuration in general position by using a specific type of configuration in general position, that of cyclic configurations.

 \section{Preliminaries} \label{section:prelim}
Let $A=(\mathbf{a}_1, \ldots, \mathbf{a}_n)$ be an integer matrix in $\mathbb{Z}^{m \times n}$, with column vectors $\mathbf{a}_1,\ldots,\mathbf{a}_n$ and such that $\Ker_{\ZZ}(A)\cap \NN^n=\{{\mathbf{0}}\}$.
The toric ideal of $A$ is the ideal $I_A\subset K[x_1,\ldots,x_n]$
  generated by the binomials ${x}^{\mathbf{u}^+}-{x}^{{\mathbf{u}}^-}$ 
 where $K$ is a field, ${\mathbf{u}}\in\Ker_{\mathbf{Z}}(A)$
and ${\bf u}={\bf u}^+-{\bf u}^-$ is the unique expression of ${\bf u}$ as a difference of two nonnegative vectors with disjoint support; see \cite[Chapter 4]{St}.
 
Let $\mathbf{u}, \mathbf{v}, \mathbf{w}\in  \Ker_\mathbb{Z}(A)$, we say that $\mathbf{u}=\mathbf{v}+_{c} \mathbf{w}$ is a \textit{conformal decomposition} of the 
vector $\mathbf{u}$
 if ${\bf u}={\bf v}+{\bf w}$ and ${\bf u}^+={\bf v}^++{\bf w}^+, {\bf u}^-={\bf v}^-+{\bf w}^-$. The conformal decomposition is called proper if both
 ${\bf v}$ and  ${\bf w}$ are not zero.
For the conformality, in terms of signs  coordinate-wise, the corresponding notation is the following:
    $+ = \oplus +_{c} \oplus$, 
    $- = \ominus+_{c} \ominus$,
    $ 0 \ = \ 0 +_{c}  0 $.
 where the symbol $\ominus $ means that the corresponding integer is nonpositive and the symbol $\oplus $ nonnegative.
By $\Gr(A)$, we denote the set of elements in $\Ker_\mathbb{Z}(A)$ that do not have a proper conformal decomposition. 
 A binomial ${\bf x}^{{\bf u}^+}-{\bf x}^{{\bf u}^-}\in I_A$
 is called \textit{primitive} if ${\bf u}\in \Gr(A).$ The set of the
primitive binomials is finite and it is called the \textit{Graver basis} of $I_A$ and is denoted by $\Gr(I_A)$, \cite[Chapter 4]{St}. 

We recall from \cite[Definition 3.9]{HS} that for vectors ${\bf u},{\bf v},{\bf w}\in\Ker_{\mathbb{Z}}(A)$ such that ${\bf u}={\bf v}+{\bf w}$, the sum is said to be a \textit{semiconformal decomposition} of ${\bf u}$, written ${\bf u}={\bf v}+_{sc} {\bf w}$, if $v_i>0$ implies that $w_i\geq 0$, and $w_i<0$ implies that $v_i\leq 0$, for all $1\leq i\leq n$.  The decomposition is called {\it proper} if both ${\bf v}, {\bf w}$ are nonzero. The set of \textit{indispensable elements} $S(A)$ of $A$ consists  of all nonzero vectors in $\Ker_{\ZZ}({A})$ with no proper semiconformal decomposition. 
For the semiconformality, in terms of signs coordinate-wise, the corresponding notation is the following:
 $\oplus = * +_{sc} \oplus $, $ \ominus =  \ominus +_{sc}  * $, where the symbol $*$ means that it can take any value.

 A binomial ${\bf x}^{{\bf u}^+}-{\bf x}^{{\bf u}^-}\in I_A$
 is called \textit{indispensable} binomial if it belongs to the intersection of all minimal systems of binomial generators of $I_A$,  up to identification of opposite binomials. The set of indispensable binomials is  $S(I_A)=\{{\bf x}^{{\bf u}^+}-{\bf x}^{{\bf u}^-}|{\bf u}\in S(A)\}$ by \cite[Lemma 3.10]{HS} and \cite[Proposition 1.1]{CTV}.

  {\em Circuits} are irreducible binomials of a toric ideal $I_A$ with minimal support. In vector notation, a vector ${\bf u}\in \Ker_{\ZZ}(A)$ is called a circuit of the matrix $A$ if supp$({\bf u})$ is minimal and the components of ${\bf u}$ are relatively prime.

To the vectors $\mathbf{a}_1,\ldots,\mathbf{a}_n$ one can associate the oriented vector matroid $M_A$ (see \cite{BLSWZ} for details). The support of a vector ${\bf v}\in \ZZ^n$ is the set supp$({\bf v})=\{i|v_i\not=0\}\subset \{1, \dots, n\}$. A {\em co-vector} is any  vector of the form $({\bf u} \cdot {\bf a}_1, \dots,{\bf u} \cdot {\bf a}_n)$, where  ${\bf u}\in \ZZ^m$. A {\em co-circuit} of $A$ is any non-zero co-vector of minimal support.  A co-circuit with support of cardinality one is called a {\em co-loop}. We call the vector ${\bf a}_i$ {\em free}  if $\{i\}$ is the support of a co-loop.  A free vector ${\bf a}_i$  belongs to any basis of the matroid $M_A$. 

 Let $E_A$ be the set consisting of elements of the form $\{ {\bf a}_i, {\bf a}_j\}$ such that there exists a co-vector ${\bf c}_{ij}$ with support  $\{i, j\}$. We denote by $E_A^+$  the subset of $E_A$ where the co-vector is a co-circuit and the signs of the two nonzero components of ${\bf c}_{ij}$ are distinct, and we denote by $E_A^-$ the subset of $E_A$ where the co-vector is a co-circuit and the signs of the two nonzero components of ${\bf c}_{ij}$ are the same. Furthermore, we denote by $E_A^0$ the subset of $E_A$ where the co-vector is not a co-circuit. This implies that both ${\bf a}_i$ and ${\bf a}_j$ are free vectors. The three sets $E_A^+,  E_A^-,  E_A^0$ partition $E_A$.

 The {\em bouquet graph} $G_A$ of $I_A$ is the graph whose vertex set  is $\{{\bf a}_1,\dots, {\bf a}_n\}$ and edge set  $E_A$.  The {\em bouquets} of $A$ are the connected components of $G_A$.  If there are free vectors in $A$ they form one bouquet with all edges in $E_A^0$.  A non-free bouquet is called {\em mixed} if it contains at least an edge from $E_A^-$, and {\em non-mixed} if it is either an isolated vertex or all of its edges are from $E_A^+$.

The vectors in the same bouquet $B_i$ have parallel Gale transforms. The relations between the Gale transforms of the vectors of the same bouquet $B_i$ define a vector $c_{B_i}$. These vectors $c_{B_i}$ together with the elements ${\bf a}_j\in B_i$ define a vector $\ab_{B_i}$ for each bouquet $B_i$.    
Let $A_B$  be the matrix with columns the vectors $\ab_{B_i}$, $i\in [s]=\{j| 1\leq j\leq s\}\subset \mathbb{Z}$, then the toric ideal $I_{A_B}$ is called  the bouquet ideal  of $A$, for details see \cite[Section 1]{PTV}.

 A toric ideal is called \textit{simple} if every bouquet is a singleton, in other words if $I_T\subset K[x_1,\dots , x_s]$ and has $s$ bouquets. Note that the bouquet ideal is simple. The bouquet ideal of a simple toric ideal $I_A$ is $I_A$. Non principal toric
ideals of monomial curves  are simple, see \cite{KTV2}.

 \begin{Definition}
 \label{TomegaIdeal}
 Let $I_T\subset K[x_1,\dots , x_s]$ be a simple toric ideal, $T=[\tb_1,\ldots,\tb_s]$ and $\omega \subset \{1, \dots , s\}$. A toric ideal $I_A$
 is called $T_{\omega}$-ideal if and only if
 \begin{itemize}
  \item the bouquet ideal of $I_A$ is $I_T$ and
  \item $\omega =\{i\in [s]| B_i \ \text{is non-mixed}\}.$
 \end{itemize}
 \end{Definition}

By $\Lambda (T)$ we denote the second Lawrence lifting of $T$, which is the $(m+s)\times 2s$ matrix
${\begin{pmatrix} T & 0 \\
I_s & I_s 
\end{pmatrix}}.$ The map $D: \Ker_{\mathbb{Z}}(T)\longrightarrow\Ker_{\mathbb{Z}}(\Lambda (T))$ given by $D(\ub)=(\ub,-\ub)$ defines an isomorphism. By $\Lambda (T)_{\omega}$ we denote the matrix taken from $\Lambda (T)$ by removing the $(m+i)$-th row and the $(s+i)$-th column for each $i\in \omega$. In the case that $T$ has no free vector, the ideal $I_{\Lambda (T)_{\omega}}$ is a $T_{\omega}$-ideal and it has $s$ bouquets, the $|\omega|$ are non-mixed and the $s-|\omega|$ are mixed. By $|\omega|$ we denote the cardinality of the set $\omega$. The map $D_{\omega}: \Ker_{\mathbb{Z}}(T)\longrightarrow\Ker_{\mathbb{Z}}(\Lambda (T))_{\omega}$ given by $D_{\omega}(\ub)=(\ub,-[{\bf u}]^{\omega})$ defines an isomorphism,
 where $[{\bf u}]^{\omega}$ is the vector ${\bf u}$ with the $i^{th}$ component missing, for $i \in \omega$. The map $D_{\omega}$ provides a bijective correspondence between the Graver basis
 of $T$ and the Graver basis of $\Lambda (T)_{\omega}$, see \cite[Theorem 1.1]{PTV}. Explicitly: $\Gr(\Lambda (T)_{\omega})=\{D_{\omega}(\ub)|\ub \in \Gr(T) \}.$

 The next proposition generalizes Proposition 4.1 of \cite{KTV2}.

 \begin{Proposition} \label{conformal} Let $T$ be a simple configuration and  ${\bf u},{\bf v},{\bf w}\in\Ker_{\mathbb{Z}}(T)$. If  $D_{\omega}({\bf u})=D_{\omega}({\bf v})+_{sc}D_{\omega}({\bf w})$ in $\Ker_{\mathbb{Z}}\left(\Lambda (T)_{\omega}\right)$, then $[{\bf u}]^{\omega}=[{\bf v}]^{\omega}+_{c}[{\bf w}]^{\omega}$, where $[{\bf u}]^{\omega}$ is the vector ${\bf u}$ with the $i^{th}$ component missing, for $i \in \omega$. 
\end{Proposition}
\begin{proof}
    Let $j \in [s] \text{ such that } j\not\in \omega$. Then, for the vector $D_{\omega}({\bf u})$ in the kernel $\Ker_{\mathbb{Z}}\left(\Lambda (T)_{\omega}\right)$, one of the components is equal to $u_j$ and another is $-u_j$. Similarly, the corresponding two components of each of $D_{\omega}({\bf v})$ and $ D_{\omega}({\bf w}) \in \Ker_{\mathbb{Z}}\left(\Lambda (T)_{\{i\}}\right)$ are 
    $v_j, -v_j$ and $w_j, -w_j$ respectively. The semiconformal decomposition $D_{\omega}({\bf u})=D_{\omega}({\bf v})+_{sc}D_{\omega}({\bf w})$, implies that  on those components we have 
   \begin{eqnarray}\label{sc_1}
 (u_j) & = & (v_j)+_{sc} (w_j), \hspace{0.5cm} \\
\label{sc_2}
   (-u_j) & = & (-v_j)+_{sc} (-w_j). \hspace{0.5cm}
  \end{eqnarray}

    If $u_j\geq 0$, then the signs of (\ref{sc_1}) are  $\oplus = * +_{sc} \oplus $ and the signs of (\ref{sc_2}) are $ \ominus =  \ominus +_{sc}  * $ then $w_j\geq 0$, while $-v_j\leq 0$. Therefore, both $v_j, w_j$ are non-negative and so the sum $(u_j)=(v_j)+_{c} (w_j)$ is conformal. If on the other hand $u_j\leq 0$, then the signs of (\ref{sc_1}) are $ \ominus =  \ominus +_{sc}  * $  and the signs of (\ref{sc_2}) are $\oplus = * +_{sc} \oplus $ then $v_j\leq 0$  and $-w_j\geq 0$. Therefore, both $v_j, w_j$ are non-positive and the sum $(u_j)=(v_j)+_{c} (w_j)$
    is again conformal.  \qed
    
\end{proof}

In \cite{KTV}, Kosta et al.  introduced a simplicial complex, which determines the strongly robust property for toric ideals. In the sence that if you have a simple toric ideal $I_T$
for which you know the simplicial complex then you can construct infinitely many strongly robust toric ideals. And if you have any strongly robust toric ideal $I_A$ then there exists a simple toric ideal $I_T$ such that the bouquet ideal of $I_A$ is $I_T$ and the set of indices 
$\omega$ such that the corresponding bouquet of $I_A$ is non-mixed belongs to the strongly robust simplicial complex of $T$.

To simplify the presentation of the current article, we give an equivalent but simpler definition of this simplicial complex, based on \cite[Theorems 3.6, 3.7]{KTV}.

\begin{Definition}  The set $\omega$ belongs to the simplicial complex $\Delta _T$ if and only if $I_{\Lambda (T)_{\omega}}$ is strongly robust.
\end{Definition}

 The set $\Delta _T$ is called the {\it strongly robust complex} of $T$ and according to \cite[Corollary 3.5]{KTV}, $\Delta _T$ is a simplicial complex. The $\Delta _T$ determines the strongly robust property for toric ideals,  by \cite[Theorem 3.6]{KTV}, since  the toric ideal $I_A$ is strongly robust if
and only if $\omega$ is a face of the strongly robust complex  $\Delta _T$. This means
also that if a $T_{\omega}$-ideal $I_A$ is strongly robust then all $T_{\omega}$-ideals
are strongly robust.

Given two simplicial complexes $\Delta '\subset \Delta$, we say that $\Delta '$ is an induced subcomplex of $\Delta$
if every simplex in $\Delta$ with all vertices in $\Delta '$  is a simplex in $\Delta '$  as well. In particular, if $\Delta$ is a simplicial complex with vertex set $\Sigma$ and $\sigma\subset \Sigma$ we say that $[\Delta ]_{\sigma}=\{\omega\cap \sigma|\omega \in \Delta\}$
 is the induced simplicial subcomplex of $\Delta $ on $\sigma $.

     \begin{Proposition} \label{subcomplex}
         Let $T'=\{\tb _i|i\in \sigma \subset [s]\}\subset T$ and both $T', T$ be simple configurations. Let $[\Delta _T]_{\sigma}$ be
         the induced simplicial subcomplex of $\Delta _T$ on $\sigma $, then $$[\Delta _T]_{\sigma}\subseteq \Delta _{T'}.$$
     \end{Proposition}

     \begin{proof}
         Let $\omega\cap \sigma$ be an element of the induced simplicial subcomplex $[\Delta _T]_{\sigma}$ of $\Delta _T$ on $\sigma$, where $\omega \in \Delta _T$.  We claim that $\omega\cap \sigma \in \Delta _{T'}$. If $\omega\cap \sigma $ were not a face of $\Delta _{T'}$, then $I_{\Lambda (T')_{\omega\cap \sigma}}$ would not be strongly robust by definition. That means there would exist an element $\ub \in \Gr (\Lambda (T')_{\omega\cap \sigma}) $ which would not be indispensable in $\ker_{\mathbb{Z}} (\Lambda (T')_{\omega\cap \sigma}) $.  This would mean that $\ub$ has a proper semi-conformal decomposition $\ub = \vb +_{sc} \wb$, where $\ub, \vb, \wb \in \ker_{\mathbb{Z}}(\Lambda (T')_{\omega\cap \sigma})$. 
         For an element $\ub\in \ker_{\mathbb{Z}} (\Lambda (T')_{\omega\cap \sigma}) $ we denote $\tilde{\ub} = (\ub, 0) $ an element $ \ker_{\mathbb{Z}}(\Lambda (T)_{\omega}) $ with $\tilde{u}_i=u_i$ if $i\in \sigma$
         and $\tilde{u}_i=0$ if $i\in [s]-\sigma$.
         Then the element $\tilde{\ub} = (\ub, 0) \in Gr(\Lambda (T)_{\omega}) $ by Proposition 4.13 in \cite{St}. Let $\tilde{\vb} = (\vb, 0)$ and $\tilde{\wb} = (\wb, 0)$ then $\tilde{\ub} = \tilde{\vb} +_{sc} \tilde{\wb} $ is a proper semi-conformal decomposition.  However, since $\omega \in \Delta_T$, we have that the ideal $I_{\Lambda (T)_{\omega}}$ is strongly robust by the definition of the strongly robust complex. Then the element $\tilde{\ub}$ of the Graver basis $\Gr(\Lambda (T)_{\omega})$ is indispensable and so $\tilde{\ub}$ cannot have a proper semi-conformal decomposition, which is a contradiction. Thus $\omega\cap \sigma \in \Delta _{T'}$.
     \end{proof}

In \cite[Corollary 1.3]{S}, Sullivant proved that strongly robust codimension 2 toric ideals 
 have at least 2 mixed bouquets. For the strongly robust complex $\Delta _T$, this result means that $\dim(\Delta _T) < s-2.$ If Sullivant's conjecture holds, namely that for every simple codimension $r$ toric ideal $I_T$ we have $\dim(\Delta_T) < s-r = \rank(T)$, then the following example shows that $\rank(T) - 1$ is the best possible upper bound for each $m = \rank(T)$. 
 \begin{Example}
     {\em Let 
     $ n_1, n_2,n_3$ be such that  $I_{(n_1, n_2,n_3)}$ is a  complete intersection on $n_3$, that means that $I_{(n_1, n_2,n_3)}$ is complete intersection and  $c_1n_1= c_2n_2\not =c_3n_3$, see \cite[Definition 3.1]{KTV2}. Namely, $I_{(n_1, n_2,n_3)}$ is minimally generated by two binomials with  different Betti degrees. By \cite[Theorem 4.8]{KTV2}, we have  $\Delta _{(n_1, n_2,n_3)}=\{\emptyset, \{3\}\}$. Let $T$ be the $m\times 3m$ matrix of rank $m$ with $t_{i,3i-2}=n_1$, $t_{i,3i-1}=n_2$, 
     $t_{i,3i}=n_3$ for $1\leq i\leq m$ and all the other $t_{i,j}=0. $ Thus, every column has one non zero element. We claim that 
     $I_T$ is simple. Simple toric ideals are those that all their bouquets are singletons and by the definition of the bouquet two  $\tb _j$, $\tb _k$  belong to the same bouquet if there exists a covector $\cb \in \mathbb{Z}^m$ such that $\cb \cdot \tb _j\not=0$, $\cb \cdot \tb _k\not=0$ and
     $\cb \cdot \tb _l=0$ for all other $l$. But note that, for any $\cb \in \mathbb{Z}^m$,
     if one of $\cb \cdot \tb _{3i-2}=c_in_1$, $\cb \cdot \tb _{3i-1}=c_in_2$, $\cb \cdot \tb _{3i}=c_in_3$  
     is different from zero, then all are different from zero. Therefore, there are not two vectors in the same bouquet and, thus, all bouquets are singletons and $I_T$ is simple. 
     
     Next, we claim that $\Delta _T=3\cdot 2^{[m]}$, where $2^{[m]}$ is the set of all subsets of $[m]$. Note that if $\ub \in \Ker_{\mathbb{Z}}T$,  then $(u_{3i-2}, u_{3i-1}, u_{3i})\in \Ker_{\mathbb{Z}}(n_1, n_2,n_3)$ for all $i$. We claim that if $\ub \in \Gr (T)$, then there exist an $i$ such that $(u_{3i-2}, u_{3i-1}, u_{3i})\in \Gr (n_1, n_2,n_3)$ and for all $j\not=i$ we have $(u_{3j-2}, u_{3j-1}, u_{3j})=(0,0,0)$. Suppose that there exist 
$j\not=i$ such that $(u_{3i-2}, u_{3i-1}, u_{3i})\not=(0,0,0)$, $(u_{3j-2}, u_{3j-1}, u_{3j})\not=(0,0,0)$. Let $\vb\in \Ker_{\mathbb{Z}}T$ be the element with
$v_{3i-2}=u_{3i-2}, v_{3i-1}= u_{3i-1}, v_{3i}=u_{3i}$ and all other components zero. Then the sum $\ub=\vb+(\ub-\vb)$ is a proper conformal decomposition of $\ub$ since 
$\vb \not= {\bf 0} \not= (\ub-\vb)$ and the sums $0+x, x+0$ are always conformal. Thus for $\ub \in \Gr (T)$ we have that there exist an $i$ such that $(u_{3i-2}, u_{3i-1}, u_{3i})\in \Ker_{\mathbb{Z}}(n_1, n_2,n_3)$ and for all $j\not=i$ we have $(u_{3j-2}, u_{3j-1}, u_{3j})=(0,0,0)$. The claim  that $(u_{3i-2}, u_{3i-1}, u_{3i})\in \Gr (n_1, n_2,n_3)$ follows by Proposition 4.13 in \cite{St}.

     Let $\omega \in \Delta _T$ and $T_i=\{\tb_{3i-2}, \tb_{3i-1}, \tb_{3i}\}$, 
     then $T_i$ is simple. Note that $\Delta _{T_i}=\{\emptyset, \{3i\}\}$, since $\Delta _{(n_1, n_2,n_3)}=\{\emptyset, \{3\}\}$. Therefore, by Proposition \ref{subcomplex}, if 
     $\sigma=\{3i-2, 3i-1, 3i\}$, then $\omega\cap \sigma=\{3i\}$ or $\omega\cap \sigma=\emptyset$. Thus, $\omega \in 3\cdot 2^{[m]}$. So $\Delta _T\subset 3\cdot 2^{[m]}$.

Let $\Omega =3\cdot [m]$. An element in $ \Gr(\Lambda(T)_{\Omega})$  is in the form $D_{\Omega}(\ub)$ for an element  $\ub \in \Gr (T)$, by Theorem 1.11 of \cite{PTV}.  The  form of the elements of $ \Gr (T)$ implies that  $D_{\Omega}(\ub)=(\ub,-[{\bf u}]^{\Omega})=$ $$=(0,0,0,\cdots, 0,0,0, u_{3i-2},u_{3i-1},u_{3i},0,0,0,\cdots, 0,0, -u_{3i-2},-u_{3i-1},0,0, \cdots, 0,0), $$ for exactly one $i$ and  $(u_{3i-2}, u_{3i-1}, u_{3i})\in \Gr (n_1, n_2,n_3)$. 

Then $$( u_{3i-2},u_{3i-1},u_{3i}, -u_{3i-2},-u_{3i-1})\in Gr(\Lambda(T_i)_{\{3\}}) $$ by Proposition 4.13 in \cite{St}. 
     Since the toric ideal $I_{\Lambda(T_i)_{\{3\}}}$ is strongly robust we have that
      $( u_{3i-2},u_{3i-1},u_{3i}, -u_{3i-2},-u_{3i-1}) $
 is indispensable in $\ker(\Lambda(T_i)_{\{3\}})$.
     Then we claim that $D_{\Omega}(\ub)$
 is indispensable in     $\Ker{(}\Lambda(T)_{\Omega})$. 
 
 Suppose that $D_{\Omega}(\ub)=D_{\Omega}(\vb)+_{sc}D_{\Omega}(\wb)$ for some $\vb, \wb \in \ker(T)$. Then by Proposition \ref{conformal} we have $[{\bf u}]^{\Omega}=[{\bf v}]^{\Omega}+_{c}[{\bf w}]^{\Omega}$.
 
 But $\ub=(0,0,0,\cdots, 0,0,0, u_{3i-2},u_{3i-1},u_{3i},0,0,0,\cdots, 0,0,0), $
     therefore $[{\bf u}]^{\Omega}=(0,0,,\cdots, 0,0,, u_{3i-2},u_{3i-1},0,0,,\cdots, 0,0). $
 But the only conformal representation of $0$ is $0+0$ thus we conclude that 
     $[{\bf v}]^{\Omega}=(0,0,\cdots, 0,0,, v_{3i-2},v_{3i-1},0,0,\cdots, 0,0) $ and
     $[{\bf w}]^{\Omega}=(0,0,,\cdots, 0,0, w_{3i-2},w_{3i-1},0,0,\cdots, 0,0). $
     Then for $j\not=i$ we have $(v_{3j-2},v_{3j-1},v_{3j})=(0,0,v_{3j})$ and
     $(w_{3j-2},w_{3j-1},w_{3j})=(0,0,w_{3j})$. 
     
     Since both $(v_{3j-2},v_{3j-1},v_{3j}), (w_{3j-2},w_{3j-1},w_{3j})$ belong to $\Ker_{\mathbb{Z}}(n_1, n_2,n_3)$ we have $v_{3j}=w_{3j}=0$. But then from $D_{\Omega}(\ub)=D_{\Omega}(\vb)+_{sc}D_{\Omega}(\wb)$
     we have
     $$ ( u_{3i-2},u_{3i-1},u_{3i}, -u_{3i-2},-u_{3i-1})=$$ $$= ( v_{3i-2},v_{3i-1},v_{3i}, -v_{3i-2},-v_{3i-1})+_{sc}( w_{3i-2},w_{3i-1},w_{3i}, -w_{3i-2},-w_{3i-1}).
     $$
But $( u_{3i-2},u_{3i-1},u_{3i}, -u_{3i-2},-u_{3i-1}) $
 is indispensable in $\ker(\Lambda(T_i)_{\{3\}})$ therefore one of $\vb$, $\wb$ is zero.
 Therefore $D_{\Omega}(\ub)$ does not have a proper semiconformal decomposition and thus it
 is indispensable in     $\Ker{(}\Lambda(T)_{\Omega})$. We conclude that
 $\Lambda(T)_{\Omega}$ is strongly robust and thus $\Omega \in\Delta _T$. 
 
 Thus $\Delta _T$ is the m-simplex $3\cdot 2^{[m]}$ and so $\dim(\Delta _T)=m-1=\rank(T)-1$.
    Thus for any $m=\rank(T)$ we can find a simple toric ideal $I_T$ such that  $\dim(\Delta _T)=\rank(T)-1$. 
}
     
 \end{Example}

 \section{Configurations in general position and  the dimension of the strongly robust complex}
\label{dimComplex}

The main result of the article is 
 Theorem \ref{dimension} of this section which confirms Sullivant's conjecture for
 the simple toric ideals of configurations in general position.
 
 We consider configurations of vectors $A=\{a_1, \dots, a_n\}\subset \mathbb{Z}^d$ such that the cone $pos_{\mathbb{Q}}(A)$ has a vertex. In this case $\Ker_{\ZZ}(A)\cap \NN^n=\{{\mathbf{0}}\}$. Note that if $\Ker_{\ZZ}(A)\cap \NN^n\not=\{{\mathbf{0}}\}$ the sets: any minimal system of binomial generators, any reduced Gr{\"o}bner basis, the universal Gr{\"o}bner basis and the Graver basis can never be simultaneously 
 equal by \cite[Theorem 4.18]{CTV1} and thus the toric ideal $I_A$ cannot be strongly robust. 

\begin{Definition} A configuration $A=\{a_1, \dots, a_n\}\subset \mathbb{Z}^d$  is called in general position if every
  $d$ elements are linearly independent in $\mathbb{Q}^d$, where $n\ge d+2$. 
    
\end{Definition}  

Note that toric ideals of monomial curves correspond to configurations in general position. 
Another famous class of configurations in general position are the cyclic configurations, see \cite{Z} and the next section.

Next theorem proves that configurations in general position have the property that they are simple. And since subsets with more than or equal to $d+2$ elements are also in general position they are also simple. A configuration $A$ is {\em simple} if the corresponding toric ideal $I_A$ is simple.

  \begin{Theorem} \label{simple}
      Let $A$ be a configuration in general position then every subset $B$ of $A$ with more than or equal to $d+2$ elements is simple.
  \end{Theorem}
  \begin{proof} Suppose that there is a subset $B$ with at least $d+2$ elements which is not simple. This would mean that there is one bouquet of the configuration $B$ which is not a singleton. Let $a_i, a_j$ be two different elements of this bouquet. Then there exist a covector with support $i, j$.
  This means that there is a hyperplane $H_{ij}$ that passes through all other elements of $B$ except for these two  $a_i, a_j$. Let $H$ be any hyperplane in $\mathbb{Q}^d$, then it contains at most $d-1$ elements from $B$, since any $d$ elements span the whole space as they are linearly independent. Thus outside any hyperplane $H$ there are three or more elements of $B$, which contradicts the existence of the hyperplane $H_{ij}$. Thus all bouquets of $B$ are singletons and so $B$ is simple. 
      
  \end{proof}

  \begin{Remark} {\em The proof of Theorem \ref{simple} shows that toric ideals
   of a configuration in general position do not have free vectors, since there is no cocircuit with a single support. Thus also toric ideals with bouquet ideal the toric ideal of a configuration in general position  do not have  a free bouquet.}
     
  \end{Remark}

The following Theorem generalizes  Theorem 3.4 of \cite{KTV2}, since $T=(n_1, \dots ,n_s)$, $ s\ge 3$ defines a configuration in general position. Even more answers affirmatively Sullivant's question for toric ideals of a configuration in general position.
The proof is based on Proposition \ref{subcomplex} and on Sullivant's main result in \cite{S}, Corollary 1.3.
\begin{Theorem} \label{dimension} Let $T$ be a configuration in general position. For the simple toric ideal $I_T$, we have that $\dim(\Delta _T)<\rank(T).$
\end{Theorem}

\begin{proof} Suppose that $\dim(\Delta _T)\ge d$, then $\Delta _T$ contains a face $\omega$ of cardinality $d+1$.  
Consider any other element $\tb_{i}\in T, i\notin \omega$ and consider the  configuration 
$T'=\{ \tb _j|j\in \sigma=\omega\cup \{i\}\}$. The configuration $T'$ has $d+2$ elements thus it is simple by Theorem \ref{simple}.  By Proposition \ref{subcomplex}, we have that $[\Delta _T]_{\sigma}\subseteq \Delta _{T'}.$
But $\omega \in \Delta _T$ and $\omega\subset \sigma$ thus $\omega \in \Delta _T'$. Thus $I_{\Lambda (T')_{\omega}}$ is strongly robust.
The ideal $I_{\Lambda (T')_{\omega}}$ is of codimension 2 and has only one mixed bouquet, the one corresponding to $\tb_{i}$, since in $\sigma$ the only element not in $\omega$ is $i$. By Corollary 1.3 \cite{S} if a codimension 2 toric ideal $I_A$ is strongly robust then $A$ has at least two
mixed bouquets.
Therefore the ideal $I_{\Lambda (T')_{\omega}}$ is not strongly robust,  a contradiction. \qed

\end{proof}

Note that configurations in general position are simple thus  all their bouquets are non-mixed, since they are singletons. Then a toric ideal $I_T$
of a configuration in general position is a $T_{[s]}$-toric ideal, where $s$ is the cardinality of $T$. If $I_T$
was strongly robust then $\omega=[s]\in \Delta _T$ thus $s-1\leq \dim(\Delta _T)<\rank(T)=d\leq s-2,$ which would be a contradiction. Therefore
one of the implications of Theorem \ref{dimension} is that toric ideals of configurations in general position are never strongly robust. 
 This observation gives an affirmative answer to a question posed by  Petrovi{\' c} et al. \cite{PTV},  concerning whether every strongly robust toric ideal $I_A$ must necessarily admit a mixed bouquet in the
  case of configurations in general position.

\section{Toric ideals of Cyclic configurations as bouquet ideals}
\label{dimHIGH}

Although, as we saw at the end of the previous section, toric ideals of configurations in
general position are never strongly robust, the knowledge of the strongly robust comlex $\Delta_{T}$  and the theory of generalized Lawrence matrices developed in \cite[Section 2]{PTV} provide a way to produce families of examples of strongly robust toric ideals that have bouquet ideal the ideal $I_T$ of a configuration $T$ in general position. Take, for example, the matrix 
\[ 
T^5_{[7]}=\left( \begin{array}{ccccccc}
 1 & 1 & 1 & 1 & 1 & 1 & 1 \\ 1 & 2 & 3 & 4 & 5 & 6 & 7\\ 1 & 4 & 9 & 16 & 25 & 36 & 49 \\
 1 & 8 & 27 & 64 & 125 & 216 & 343\\
 1 & 16 & 81 & 256 & 625 & 1296 & 2401 
\end{array} \right). \]

The set of columns of this matrix is a particular case of a cyclic configuration.  Cyclic configurations are well known for their extremal properties, see \cite{MM, Z}. 
Let $A$ be a cyclic configuration formed by the columns of the $d\times n$  matrix \[ 
\left( \begin{array}{cccc}
  a(t_1) & a(t_2) & \hdots & a(t_{n}) 
\end{array} \right), \] where $ a(t)=
\left( \begin{array}{ccccc}
 1 & t & t^2 & \dots & t^{d-1}
\end{array} \right)^T \in \mathbb{Z}^{d}$, $t_1< t_2< \cdots <t_{n}$ are integers and $n\ge d+2$.
  Then any subset $B=\{a(t_{i_1}), \cdots, a(t_{i_d})\}$ of $A$ consisting of $d$ vectors is linearly independent, since  
the determinant of the  Vandermonde matrix with columns the elements of $B$ is given by $\prod_{1 \leq l<j \leq d}(t_{i_j}-t_{i_l})$, which is different from zero. Thus, cyclic configurations are configurations in general position and thus they define simple toric ideals.

Using $4ti2$, see \cite{4ti2}, we can compute the toric ideals $I_{\Lambda (T^5_{[7]})_{\{2\}}}, I_{\Lambda (T^5_{[7]})_{\{6\}}}, I_{\Lambda (T^5_{[7]})_{\{1,3,4,5,7\}}}$. The first two are not strongly robust, thus, $\{2\}, \{6\}$ do not belong to $\Delta _{T^5_{[7]}}$. The third ideal is strongly robust, thus, $\{1, 3, 4, 5,  7\}\in \Delta _{T^5_{[7]}}$. We conclude that 
 the simplicial complex $\Delta_{T^5_{[7]}}$ is a simplex with vertices $\{1, 3, 4, 5,  7\}.$
Thus $\dim(\Delta _{T^5_{[7]}})=4,$ which as we saw in Theorem \ref{dimension} is the maximal possible among matrices of rank 5.

Now, we follow the construction of generalized Lawrence matrices developed in \cite[Section 2]{PTV}, where one can find the details of the construction.
Choose seven integer vectors
	 of any dimension each as following.  All of the seven vectors should have full support, each vector should have the greatest common divisor of all of its components equal to $1$, and all seven vectors should have a positive first component, while the second and the sixth vector should have at least one negative component. For example, choose ${\bf c}_1=(7, 1, 2027)$, 
     ${\bf c}_2=(1, -1)$,  ${\bf c}_3=(1)$,  ${\bf c}_4=(2, 3, 7)$, 
     ${\bf c}_5=(11, 1)$,  ${\bf c}_6=(4, -1, -27)$,  and ${\bf c}_7=(1)$. For each vector  
${\bf c}_i=(c_{i1},\ldots,c_{im_i})\in\ZZ^{m_i}$, $1\leq i\leq 7$, choose 
integers $\lambda_{i1},\ldots,\lambda_{im_i}$ such that $1=\lambda_{i1}c_{i1}+\cdots+\lambda_{im_i}c_{im_i}$.  Then, the generalized Lawrence matrix
\[
{\footnotesize A=\
\left( \begin{array}{ccccccccccccccc}
 0 & 1 & 0 &  1 & 0 & 1 & -1 & 1 & 0 & 0 & 1  & 0 & -1 & 0 & 1 \\
 0 & 1 & 0 &  2 & 0 & 3 & -4 & 4 & 0 & 0 & 5  & 0 & -6 & 0 & 7 \\
 0 & 1 & 0 &  4 & 0 & 9 & -16 & 16 & 0 & 0 & 25  & 0 & -36 & 0 & 49 \\
 0 & 1 & 0 &  8 & 0 & 27 & -64 & 64 & 0 & 0 & 125  & 0 & -216 & 0 & 343 \\
 0 & 1 & 0 &  16 & 0 & 81 & -256 & 256 & 0 & 0 & 625  & 0 & -1296 & 0 & 2401 \\
 -1 & 7 & 0 &  0 & 0 & 0 & 0 & 0 & 0 & 0 & 0  & 0 & 0 & 0 & 0 \\
 -2027 & 0 & 7 &  0 & 0 & 0 & 0 & 0 & 0 & 0 & 0  & 0 & 0 & 0 & 0 \\
 0 & 0 & 0 &  1 & 1 & 0 & 0 & 0 & 0 & 0 & 0  & 0 & 0 & 0 & 0 \\
 0 & 0 & 0 &  0 & 0 & 0 & -3 & 2 & 0 & 0 & 0  & 0 & 0 & 0 & 0 \\
 0 & 0 & 0 &  0 & 0 & 0 & -7 & 0 & 2 & 0 & 0  & 0 & 0 & 0 & 0 \\
 0 & 0 & 0 &  0 & 0 & 0 & 0 & 0 & 0 & -1 & 11  & 0 & 0 & 0 & 0 \\
 0 & 0 & 0 &  0 & 0 & 0 & 0 & 0 & 0 & 0 & 0  & 1 & 4 & 0 & 0 \\
 0 & 0 & 0 &  0 & 0 & 0 & 0 & 0 & 0 & 0 & 0  & 27 & 0 & 4 & 0 \\
\end{array} \right) } , 
\]

defines a toric ideal $I_A$ with bouquet ideal $I_{T^5_{[7]}}$ by Theorem 2.1 of \cite{PTV}. Note that the columns of $A$ that correspond to the same vector ${\bf c}_i$ belong to the same i-th bouquet; if ${\bf c}_i$  has a negative and a positive component then the i-th bouquet is mixed and  if ${\bf c}_i$ has all componenets positive then the i-th bouquet is non-mixed, see \cite[Lemma 1.6]{PTV}. According to Theorem 3.6 of \cite{KTV}, the toric ideal $I_A$ is strongly robust, as $I_A$ is a $T^5_{{[7]}_{\{1,3,4,5,7 \}}}$-toric ideal (see Definition~\ref{TomegaIdeal}) and  $\{1, 3, 4, 5,  7\}\in \Delta _{T^5_{[7]}}$.

Corollary 2.3 of \cite{PTV}
asserts that  all  toric ideals with bouquet ideal $I_{T^5_{[7]}}$ are obtained in this way, for some appropriate seven vectors ${\bf c}_1, 
     {\bf c}_2, {\bf c}_3,  {\bf c}_4, 
     {\bf c}_5,  {\bf c}_6,  {\bf c}_7$. Actually, Corollary 2.3 of \cite{PTV} combined with Theorem 3.6 of \cite{KTV},
asserts that in fact all strongly robust toric ideals with bouquet ideal $I_{T^5_{[7]}}$ are obtained this way, for some  vectors ${\bf c}_1, 
     {\bf c}_2, {\bf c}_3,  {\bf c}_4, 
     {\bf c}_5,  {\bf c}_6,  {\bf c}_7$ with a positive first component, the rest of the components of ${\bf c}_1, 
      {\bf c}_3,  {\bf c}_4, 
     {\bf c}_5,  {\bf c}_7$  being either positive or negative and the $
     {\bf c}_2, {\bf c}_6$  having at least one negative component. In this way, the set $\omega$ of indices that correspond to non-mixed bouquets is a subset of $\{1,3,4,5,7\}$, thus, belongs to $\Delta_{T^5_{[7]}}$.

\section*{Acknowledgements}
D. Kosta gratefully acknowledges funding from the Royal Society Dorothy Hodgkin Research Fellowship DHF$\backslash$R1$\backslash$201246.

\end{document}